\documentclass[12pt]{amsart}
\usepackage{amssymb,hyperref}

\newtheorem{theorem}{Theorem}[section]
\newtheorem{definition}[theorem]{Definition}

\newtheorem{corollary}[theorem]{Corollary}

\begin{document}

\author{Teodor Banica}
\address{T.B.: Department of Mathematics, Cergy-Pontoise University, 95000 Cergy-Pontoise, France. {\tt teodor.banica@u-cergy.fr}}

\author{Uwe Franz}
\address{U.F.: Department of Mathematics, Universit\'e de Franche-Comt\'e, 16 route de Gray, 25030 Besan\c con cedex, France. {\tt uwe.franz@math.univ-fcomte.fr}}

\author{Adam Skalski}
\address{A.S.: Institute of Mathematics of the Polish Academy of Sciences, ul. \'Sniadeckich 8, 00-956 Warszawa, Poland. {\tt a.skalski@impan.pl}}

\title{Idempotent states and the inner linearity property}


\begin{abstract}
We find an analytic formulation of the notion of Hopf image, in terms of the associated idempotent state. More precisely, if $\pi:A\to M_n(\mathbb C)$ is a finite dimensional representation of a Hopf $C^*$-algebra, we prove that the idempotent state associated to its Hopf image $A'$ must be the convolution Ces\`aro limit of the linear functional $\varphi=tr\circ\pi$. We discuss then some consequences of this result, notably to inner linearity questions.
\end{abstract}

\maketitle

\section*{Introduction}

The compact quantum groups were axiomatized by Woronowicz in \cite{wo2}. The idea is that such a quantum group $G$ is an abstract object, which cannot be viewed as a set, but which is described by a well-defined algebra $A=C(G)$, which must be a Hopf $C^*$-algebra.

Woronowicz's axiomatization covers as well the discrete quantum groups. Indeed, associated to $A=C(G)$ is the discrete quantum group $\Gamma=\widehat{G}$ given by $A=C^*(\Gamma)$. Once again, $\Gamma$ is an abstract object, not a set. See \cite{wo2}.

This joint axiomatization proved to be fruitful for a number of purposes. Consider indeed a Hopf $C^*$-algebra $A$, and a representation $\pi:A\to M_n(\mathbb C)$:
\begin{enumerate}
\item We can call ``Hopf image'' of $\pi$ the smallest Hopf $C^*$-algebra quotient $A\to A'$ producing a factorization $\pi:A\to A'\to M_n(\mathbb C)$.

\item We can say that $\pi$ is ``inner faithful'' if $A=A'$, and call $A$ ``inner linear'' if it has at least  one inner faithful representation.
\end{enumerate}

The point is that in the case $A=C^*(\Gamma)$ the representation $\pi$ must come from a group representation $\pi':\Gamma\to U_n$, and we have $A'=C^*(\Gamma')$, where $\Gamma'=\pi'(\Gamma)$. Also, $\pi$ is inner faithful if and only if $\pi'$ is faithful, and $A$ is inner linear if and only if $\Gamma$ is linear.

These notions, emerging from the work in \cite{bbs}, were introduced and studied in \cite{bbi}, then in \cite{abi}. They are related to a number of key questions, coming from the Connes embedding problem for Wang's free quantum groups, and from a number of key problems regarding subfactors and Hadamard matrices. We will discuss here some of these questions.

The aim of the present paper is to develop an analytic point of view on these notions, by relating them to the theory of idempotent states, developed in \cite{fs1}, \cite{fs2}, \cite{fst}, \cite{ssk}.

Our main result, stated and proved in sections 1-2 below, will be an idempotent state formulation for the notion of Hopf image. As a consequence, we will have as well an idempotent state formulation for the notion of inner linearity, that we will further develop in section 3. Finally, in section 4 we discuss a number of open questions.

\subsection*{Acknowledgements}

The work of T.B. was supported by the ANR grant ``Granma''. U.F. was supported by the ANR grant 2011 BS01 008 01. A.S. was partly supported by the National Science Centre (NCN) grant no. 2011/01/B/ST1/05011. Part of this work was done during a visit of U.F. at the Banach Center in Warsaw in October 2011. We would like to thank the referee for the comments improving the clarity of the presentation.

\section{States and images}

A compact quantum group $G$ is an abstract object, having no points in general, but which is described by a well-defined algebra $C(G)$, which must be a Hopf $C^*$-algebra. The axioms for Hopf $C^*$-algebras, found by Woronowicz in \cite{wo2}, are as follows:

\begin{definition}
A Hopf $C^*$-algebra is a $C^*$-algebra $A$, given with a morphism of $C^*$-algebras $\Delta:A\to A\otimes A$, called comultiplication, subject to the following  conditions:
\begin{enumerate}
\item Coassociativity: $(\Delta\otimes id)\Delta=(id\otimes\Delta)\Delta$.

\item $\overline{span}\,\Delta(A)(A\otimes 1)=\overline{span}\,\Delta(A)(1\otimes A)=A\otimes A$.
\end{enumerate}
\end{definition}

The basic example is $A=C(G)$, where $G$ is a compact group, with $\Delta f(g,h)=f(gh)$. The fact that $\Delta$ is coassociative corresponds to $(gh)k=g(hk)$, and the conditions in (2) correspond to the cancellation rules $gh=gk\implies h=k$ and $gh=kh\implies g=k$.

The other main example is $A=C^*(\Gamma)$, where $\Gamma$ is a discrete group, with comultiplication $\Delta(g)=g\otimes g$. One can prove that any Hopf $C^*$-algebra which is cocommutative, in the sense that $\Sigma\Delta=\Delta$, where $\Sigma(a\otimes b)=b\otimes a$ is the flip, is of this form.

These basic facts, together with some other general results in \cite{wo2}, lead to:

\begin{definition}
Associated to any Hopf $C^*$-algebra $A$ are a compact quantum group $G$ and a discrete quantum group $\Gamma=\widehat{G}$, according to the formula $A=C(G)=C^*(\Gamma)$.
\end{definition}

The meaning of this definition is of course quite formal. The idea is that, with a suitable definition for morphisms, the Hopf $C^*$-algebras form a category $H$. One can define then the categories of compact and discrete quantum groups to be $\widehat{H}$, and $H$ itself, and these categories extend those of the usual compact and discrete groups. See \cite{wo2}.

Woronowicz's axiomatization proved to be fruitful for a number of purposes. Among others, we have the following key definition from \cite{bbi}, emerging from the work in \cite{bbs}:

\begin{definition}
The Hopf image of a representation $\pi:A\to M_n(\mathbb C)$ is the smallest Hopf $C^*$-algebra quotient $A\to A'$ producing a factorization $\pi:A\to A'\to M_n(\mathbb C)$.
\end{definition}

The last definition requires an explanation: in fact each Hopf $C^*$-algebra $A$ as defined in Definition 1.1 admits a unique dense Hopf $^*$-algebra $\mathcal{A}$, and Hopf $^*$-algebras arising in this way, called \emph{CQG-algebras}, admit an intrinsic characterisation (see \cite{dik}). When we talk above about a smallest Hopf $C^*$-algebra quotient we agree to identify $C^*$-Hopf algebras with identical underlying CQG algebra. An alternative solution would be to formulate everything in the purely algebraic language of CQG algebras (as it is in \cite{bbi}), but we prefer to stick to $C^*$-algebras to make a more direct connection to open problems discussed in Section 4.

In order to understand the notivation behind the  notion of the Hopf image, let $A=C^*(\Gamma)$. Then $\pi$ must come from a unitary group representation $\pi':\Gamma\to U_n$, and we have $A'=C^*(\Gamma')$, where $\Gamma'=\pi'(\Gamma)$.

In the above computation $\Gamma$ was of course a usual discrete group. In the general case, i.e. when $\Gamma$ is a discrete quantum group, it is only known that the Hopf image exists, and is unique \cite{bbi}. But, of course, the discrete quantum group point of view is very useful.

Observe also that in the case $A=C(G)$, with $G$ compact group, any representation $\pi:A\to M_n(\mathbb C)$ must come from the evaluation at some points $g_1,\ldots,g_n \in G$. Thus the Hopf image is simply $A'=C(G')$, where $G'=\overline{<g_1,\ldots,g_n>}$. See \cite{abi}, \cite{bbi}.

Here are a few more definitions from \cite{bbi}, based on the same philosophy:

\begin{definition}
Let $A$ be a Hopf $C^*$-algebra.
\begin{enumerate}
\item A representation $\pi:A\to M_n(\mathbb C)$ is called inner faithful if $A=A'$.

\item $A$ is called inner linear if it has an inner faithful representation.
\end{enumerate}
\end{definition}

Observe that with $A=C^*(\Gamma)$, and with the above notations, $\pi$ is inner faithful if and only if $\pi'$ is faithful. Also, $A$ is inner linear if and only if $\Gamma$ is linear. See \cite{bbi}. Note also that if $A$ is the universal $C^*$-completion of its underlying CQG algebra $\mathcal{A}$ (as is the case in all the examples we study in the following sections), then representations of $A$ are in a 1-1 correspondence with these of $\mathcal{A}$.

We recall now that the state space of $A$ is endowed with the convolution product $\varphi*\psi=(\varphi\otimes\psi)\Delta$. We have the following definition, from \cite{fs1}:

\begin{definition}
A state $\varphi:A\to\mathbb C$ is called idempotent if $\varphi*\varphi=\varphi$.
\end{definition}

The basic example of idempotent state is the Haar functional $h:A\to\mathbb C$. More generally, we have an idempotent state associated to any quantum subgroup $H\subset G$. Indeed, such a quantum subgroup $H\subset G$ must come from a surjective morphism of Hopf $C^*$-algebras $\pi:C(G)\to C(H)$, and we have the following definition:

\begin{definition}
Associated to any surjective Hopf $C^*$-algebra morphism $\pi:A\to B$ is the idempotent state $\varphi_B=h_B\pi$, where $h_B:B\to\mathbb C$ is the Haar functional of $B$.
\end{definition}

In the classical case one can prove that all the idempotent states come from closed subgroups, via the above construction \cite{kit}. However, in the general quantum case, this fundamental result does not hold \cite{pal}. We refer to the series of papers \cite{fs1}, \cite{fs2}, \cite{fst}, \cite{ssk} for more details on this question, and for the general theory of idempotent states.

\section{The main result}

We state and prove here our main result. We recall from Definition 1.3 above that associated to any representation $\pi:A\to M_n(\mathbb C)$ is its Hopf image $A\to A'$. The problem is to ``locate'' this Hopf image, and one concrete question in this sense is that of computing its associated idempotent state, in the sense of Definition 1.6.

In order to answer this question, we need one more definition:

\begin{definition}
We denote by $\tilde{\varphi}$ the convolution Ces\`aro limit of a state $\varphi$:
$$\tilde{\varphi}=\lim_{N\to\infty}\frac{1}{N}\sum_{k=1}^N\varphi^{*k}$$
\end{definition}

We will use the following fundamental fact, due to Woronowicz: given any faithful state $\varphi:A\to\mathbb C$, the corresponding Ces\`aro limit $\tilde{\varphi}$ coincides with the Haar functional, $\tilde{\varphi}=h$. This is indeed how the Haar functional $h:A\to\mathbb C$ can be constructed. See \cite{wo2}.

Back now to the above question, the answer is particularly simple:

\begin{theorem}
Let $\pi:A\to M_n(\mathbb C)$ be a representation of a Hopf $C^*$-algebra $A$, with Hopf image $A'$. Then the idempotent state corresponding to $A'$ is $\tilde{\varphi}$, where $\varphi=tr\circ\pi$.
\end{theorem}

\begin{proof}
Given a positive functional $f:A\to\mathbb{C}$, we denote by $N_f=\{a\in A|f(a^*a)=0\}$ its null space. This is a always left ideal. If $f$ is tracial, then $N_f$ is a two-sided $*$-ideal.

The state $\varphi=tr\circ\pi$ is by definition tracial, and since the convolution preserves traciality, $\tilde{\varphi}$ is a tracial idempotent state on $A$. Therefore, Theorem 3.3 in \cite{fst} implies that
$\tilde{\varphi}$ is a ``Haar'' idempotent state, i.e.\ it is induced by a Hopf $C^*$-algebra quotient $\pi_\varphi:A\to A_\varphi$ as $\tilde{\varphi}=h_\varphi \circ \pi_\varphi$, where $h_\varphi$ denotes the Haar state of $A_\varphi$, and ${\rm ker}\, \pi_\varphi=N_{\tilde{\varphi}}$.

In order to prove that $A'\simeq A_\varphi$, we will show that the ideals arising as kernels of the respective quotient maps are equal. By using the formulae in \cite{fst}, it is sufficient to show that the Haar states $h',h_\varphi$ of $A',A_\varphi$ induce the same idempotent state on $A$. By density it suffices to conduct the proof on the level of the corresponding CQG algebras; we will denote them respectively by $\mathcal{A}$, $\mathcal{A}'$ and $\mathcal{A}_{\varphi}$, similarly we write $\mathcal{N}_{\tilde{\varphi}}$ for $N_{\tilde{\varphi}} \cap \mathcal{A}$.

Let us first check that $\mathcal{A}_\varphi$ is a Hopf $*$-algebra quotient of $\mathcal{A}'$. In \cite{bbi} it was shown that $\mathcal{A}'$ is the quotient of $\mathcal{A}$ by the largest Hopf $*$-ideal contained in $\ker\pi|_{\mathcal{A}}$, namely:
$$I_\pi^+ =\bigcap_{r\geq 1}\bigcap_{k_1,\ldots,k_r\in\mathbb{Z}} {\rm ker}\big((\pi\circ
S^{k_1})\otimes \ldots\otimes (\pi\circ S^{k_r})\big)\circ \Delta^{(r-1)}$$

Since the trace is faithful on $M_n(\mathbb{C})$ and on its tensor products, we have:
$$\ker\left(\pi^{\otimes k}\circ \Delta^{(k-1)}\right) = N_{\varphi^{*k}}$$

We conclude that we have:
$$I^+_\pi \subseteq \bigcap_{k=1}^\infty {\rm ker}\, \pi^{\otimes
  k}\circ \Delta^{(k-1)} = \bigcap_{n=1}^\infty \mathcal{N}_{\varphi^{\star n}}
\subseteq \mathcal{N}_{\tilde{\varphi}}$$

The inclusion of the ideals yields a canonical quotient map $\mathcal{A}'\to \mathcal{A}_\varphi$ and $\tilde{\varphi}$ induces in this way an idempotent state $\tilde{\varphi}'$ on $\mathcal{A}'$, which has to satisfy $\tilde{\varphi}'*h'=h'=h'*\tilde{\varphi}'$ (note that of course we can convolve arbitrary, not necessarily continuous functionals on a Hopf $^*$-algebra). Composing these states with the
canonical quotient map $\gamma:\mathcal{A}\to \mathcal{A}'$, we get:
$$\tilde{\varphi}*(h'\circ \gamma)= (h'\circ \gamma) = (h'\circ \gamma)*\tilde{\varphi}$$

Let now $\psi:M_n(\mathbb C)\to\mathbb C$ be a positive linear functional. Then
there exists a positive matrix $Q\in M_n(\mathbb{C})$ such that $\psi(B)=tr(QB)$, for any $B\in M_n(\mathbb C)$. Therefore for all the positive matrices $B\in M_n(\mathbb{C})_+$ we have:
$$\psi(QB)=tr(Q^{1/2}BQ^{1/2}) \le ||Q|| tr(B)$$

In other words, a non-negative multiple of $tr$ dominates $\psi$ on $M_n(\mathbb{C})$. It follows that a non-negative multiple of $\varphi=tr\circ\pi$ dominates $\psi\circ\pi$ on $\mathcal{A}$. Since $\tilde{\varphi}$ is the Ces\`aro limit of the convolution powers of $\varphi$, we also have $\varphi*\tilde{\varphi}=\tilde{\varphi}=\tilde{\varphi}*\varphi$.

According now to Lemma 2.2 in \cite{van}, if $\omega,\phi:\mathcal{A}\to\mathbb{C}$ are two states such that $\omega*\phi=\phi$, and if $\rho:\mathcal{A}\to\mathbb{C}$ is a positive linear functional dominated by some non-negative multiple of $\omega$, then $\rho*\phi=\rho(1)\phi$. In our situation, as described above, this gives:
$$(\psi\circ \pi)* \tilde{\varphi} = \psi(I) \tilde{\varphi}$$

Since any linear functional on $M_n(\mathbb{C})$ can be written as a linear combination of four positive linear functionals, this must hold for all linear functionals $\psi\in M_n(\mathbb{C})^*$.

Similarly we get $\tilde{\varphi}*(\psi\circ \pi) = \psi(I)\tilde{\varphi}$ for all $\psi\in M_n(\mathbb{C})^*$. Since idempotent states are invariant under the antipode, cf. Equation (3.1) in \cite{fs1}, we can deduce also that:
\begin{equation} \label{convform}(\psi\circ\pi\circ S^k)*\tilde{\varphi}=\psi(I)\tilde{\varphi}=\tilde{\varphi}*(\psi\circ\pi\circ S^k)\end{equation}
In \cite{bbi} it is shown that clear from the construction of $\mathcal{A}'$ can be constructed as the quotient $\mathcal{A}/I^+_\pi$. Hence each functional on $\mathcal{A}'$ can be identified with a functional in the annihilator of the pre-annihilator of the set $\mathcal{C}_{\pi}$ defined in Lemma 2.3 of \cite{bbi}.
As the formula \eqref{convform} is `stable under convolution, we can replace in it $\psi\circ\pi\circ S^k$, with $k\in \mathbb{Z}$, by an arbitrary functional in $C_{\pi}$. Further Theorem 1.2.6 of \cite{dnr} shows that the annihilator of the pre-annihilator of $C_{\pi}$ is equal to the closure of $C_{\pi}$ in the so-called finite topology. Applying fundamental theorem on coalgebras and a basic limiting argument we see that in fact
\[ \omega*\tilde{\varphi}=\omega(I)\tilde{\varphi}=\tilde{\varphi}* \omega\]
for an arbitrary functional $\omega$ on $\mathcal{A}$ which vanishes on $I_{\pi}^+$. Thus $(h'\circ \gamma)*\tilde{\varphi}=\tilde{\varphi}=\tilde{\varphi}*(h'\circ \gamma)$, so $h'\circ\gamma=\tilde{\varphi}$, and we are done.
\end{proof}

\begin{corollary}
A representation $\pi:A\to M_n(\mathbb C)$ is inner faithful if and only if $\tilde{\varphi}=h$, where $\varphi=tr\circ\pi$, and $h$ is the Haar functional of $A$.
\end{corollary}

\begin{proof}
This follows from Theorem 2.2, and from the basic fact that the idempotent state associated to the identity quotient map $A\to A$ is the Haar functional of $A$.
\end{proof}

\begin{corollary}
If $A$ is a Hopf image (that is, if it is inner linear), then it satisfies the Kac algebra assumption $S^2=id$.
\end{corollary}

\begin{proof}
Since $\varphi=tr\circ\pi$ is tracial, so is the Ces\`aro limit $\tilde{\varphi}$, so the result follows from the well-known fact that $S^2=id$ if and only if the Haar functional is a trace \cite{wo2}.
\end{proof}

\section{The matrix case}

We have seen in the previous section that the notions of Hopf image and inner faithfulness from \cite{bbi} have a purely analytic formulation, in the spirit of \cite{fs1}, in terms of idempotent states and Ces\`aro limits. It is of course possible to deduce from this analytic picture a number of new proofs, sometimes simpler, for a number of algebraic results in \cite{bbi}.

In this section we will present such an application. We will directly focus on the main result in \cite{bbi}, which is a Tannakian formulation of the notion of Hopf image, and we will present here a simple, nice analytic proof, that we believe to be potentially useful.

Let us first recall the following definition, inspired from \cite{wo1}:

\begin{definition}
An orthogonal Hopf $C^*$-algebra is a $C^*$-algebra $A$, with an orthogonal matrix $u\in M_n(A)$ (i.e. $u=\bar{u}$, $u^t=u^{-1}$) whose coefficients generate $A$, such that:
\begin{enumerate}
\item The formula $\Delta(u_{ij})=\sum_ku_{ik}\otimes u_{kj}$ defines a morphism $A\to A\otimes A$.

\item The formula $\varepsilon(u_{ij})=\delta_{ij}$ defines a morphism $A\to\mathbb C$.

\item The formula $S(u_{ij})=u_{ji}$ defines a morphism $A\to A^{op}$.
\end{enumerate}
\end{definition}

The basic example of such an algebra is $A=C(G)$, where $G\subset O_n$ is a closed subgroup, with $u_{ij}(g)=g_{ij}$. The other basic example is $A=C^*(\Gamma)$, where $\Gamma=<g_1,\ldots,g_n>$ is a discrete group with generators satisfying $g_i^2=1$, with $u=diag(g_1,\ldots,g_n)$. See \cite{wo1}.

We denote by $\#(\lambda\in T)$ the multiplicity of an eigenvalue $\lambda$ of a matrix $T$. Also, we let $\chi=Tr(u)$ be the character of the fundamental corepresentation of $A$, and we denote as usual by $h:A\to\mathbb C$ the Haar functional. With these notations, our result here is:

\begin{theorem}
Let $\pi:A\to M_n(\mathbb C)$ be a representation, given by $u_{ij}\to P_{ij}$, and consider the matrix $T_k=(tr(P_{i_1j_1}\ldots P_{i_kj_k}))_{i_1\ldots i_k,j_1\ldots j_k}$. Then the following are equivalent:
\begin{enumerate}
\item $\pi$ is inner faithful.

\item $\#(1\in T_k)\leq h(\chi^k)$, for any $k$.
\end{enumerate}
\end{theorem}

\begin{proof}
This result reminds the Tannakian formulation of the inner faithfulness in \cite{bbi}, and can be deduced from it. We present below a purely analytic proof, based on Corollary 2.3 above. First, since the elements of type $u_{i_1j_1}\ldots u_{i_kj_k}$ span a dense subalgebra of $A$, the inner faithfulness of $\pi$ is equivalent to the following collection of equalities:
$$\tilde{\varphi}(u_{i_1j_1}\ldots u_{i_kj_k})=h(u_{i_1j_1}\ldots u_{i_kj_k})$$

The left term can be computed by using the fact that the multiplication Ces\`aro limit of any matrix $||T||\leq 1$ is the orthogonal projection onto its 1-eigenspace:
\begin{eqnarray*}
\tilde{\varphi}(u_{i_1j_1}\ldots u_{i_kj_k})
&=&\lim_{N\to\infty}\frac{1}{N}\sum_{r=1}^N\varphi^{*r}(u_{i_1j_1}\ldots u_{i_kj_k})\\
&=&\lim_{N\to\infty}\frac{1}{N}\sum_{r=1}^N((T_k)^r)_{i_1\ldots i_k,j_1\ldots j_k}\\
&=&\left(\lim_{N\to\infty}\frac{1}{N}\sum_{r=1}^N(T_k)^r\right)_{i_1\ldots i_k,j_1\ldots j_k}\\
&=&(Proj(1\in T_k))_{i_1\ldots i_k,j_1\ldots j_k}
\end{eqnarray*}

Regarding now the right term, we use the general philosophy in \cite{csn}. We have:
\begin{eqnarray*}
h(u_{i_1j_1}\ldots u_{i_kj_k})
&=&h((u^{\otimes k})_{i_1\ldots i_k,j_1\ldots j_k})\\
&=&((id\otimes h)(u^{\otimes k}))_{i_1\ldots i_k,j_1\ldots j_k}\\
&=&Proj(Fix(u^{\otimes k}))_{i_1\ldots i_k,j_1\ldots j_k}
\end{eqnarray*}

Here we used the basic fact from \cite{wo1} that when integrating the coefficients of a corepresentation $r$ we obtain the projection onto the space of fixed points $Fix(r)$.

Summing up, the inner linearity of $\pi$ is equivalent to the following condition:
$$(1\in T_k)=Fix(u^{\otimes k})$$

Now given a vector $\xi\in(\mathbb C^n)^{\otimes k}$, denoted $\xi=(\xi_{i_1\ldots i_k})_{i_1\ldots i_k}$, we have:
\begin{eqnarray*}
(T_k\xi)_{i_1\ldots i_k}
&=&\sum_{j_1\ldots j_k}tr(P_{i_1j_1}\ldots P_{i_kj_k})\xi_{j_1\ldots j_k}\\
&=&\sum_{j_1\ldots j_k}(tr\circ\pi)(u_{i_1j_1}\ldots u_{i_kj_k})\xi_{j_1\ldots j_k}\\
&=&(tr\circ\pi)\sum_{j_1\ldots j_k}\xi_{j_1\ldots j_k}u_{i_1j_1}\ldots u_{i_kj_k}
\end{eqnarray*}

Thus, by assuming that we have $\xi\in Fix(u^{\otimes k})$, we obtain the following formula:
$$(T_k\xi)_{i_1\ldots i_k}=(tr\circ\pi)(\xi_{i_1\ldots i_k}1)=\xi_{i_1\ldots i_k}tr(1_n)=\xi_{i_1\ldots i_k}$$

But this tells us that $\xi$ is a 1-eigenvector of $T_k$, so we obtain $(1\in T_k)\supset Fix(u^{\otimes k})$. Now by getting back to the inner linearity criterion $(1\in T_k)=Fix(u^{\otimes k})$ found above, since we have an inclusion in one sense, this criterion is equivalent to:
$$\dim(1\in T_k)\leq\dim(Fix(u^{\otimes k}))$$

Since the left term is by definition the multiplicity $\#(1\in T_k)$, and the right term is obtained by integrating the character of $u^{\otimes k}$, we obtain the result.
\end{proof}

\section{Open problems}

Let $A=C(S_n^+)$ be the quantum permutation algebra, constructed by Wang in \cite{wan}. That is, $A$ is the universal $C^*$-algebra generated by $n^2$ abstract projections $u_{ij}$, which sum up to $1$ on each row and column of $u=(u_{ij})$. We have the following questions:
\begin{enumerate}
\item Is $A$ inner linear?

\item Is $A''$ Connes-embeddable?

\item Does $A$ have an Hadamard matrix model?

\item Does $A$ have an inner faithful matrix model?
\end{enumerate}

These questions, all important, and basically open since Wang's paper \cite{wan}, are related by the sequence of implications (3)$\implies$(1)$\implies$(4)$\implies$(2). More precisely:
\begin{enumerate}
\item This is the central question. In principle Theorem 3.2 above is a good criterion here, but no candidate for such a representation is available so far.

\item Yet another central question. This is known to be slightly weaker that question (1), because of the results of Vaes in \cite{vae}.

\item This deep subfactor question, stronger than (1), is due to Jones \cite{jon}. We refer to the article \cite{ban} for the complete story here.

\item Yet another subtle question, asking this time for an inner faithful representation of type $\pi:A\to L^\infty(X)\otimes M_N(\mathbb C)$. See \cite{bco}.
\end{enumerate}

We believe that (4) is the ``good question'', and that Theorem 3.2 above can help, once a candidate for such a model is found. However, no such candidate is available so far.

\end{document}